\documentclass[11pt,psamsfonts]{amsart}
\usepackage{amsmath}
\usepackage{amsthm}
\usepackage{amssymb}
\usepackage{amscd}
\usepackage{amsfonts}
\usepackage{amsbsy}
\usepackage{enumerate}
\usepackage{graphicx}
\usepackage{microtype}
\usepackage{calc}
\usepackage{xcolor}
\usepackage{setspace}
\usepackage[dvips]{psfrag}
\usepackage{color}
\usepackage{epsfig}
\usepackage{bm}
\usepackage{url}
\usepackage{CJK}
\usepackage{mathrsfs}

\newtheorem{thm}{Theorem}[section]

\newtheorem{prop}{Proposition}[section]
\newtheorem{lem}{Lemma}[section]

\newtheorem{Cor}{Corollary}[section]

\newcommand{\be}{\begin{equation}}
\newcommand{\ee}{\end{equation}}
\newcommand\bes{\begin{eqnarray}}
\newcommand\ees{\end{eqnarray}}
\newcommand{\bess}{\begin{eqnarray*}}
\newcommand{\eess}{\end{eqnarray*}}

\title[Reduction of Elementary integrability]{Reduction of Elementary Integrability of Polynomial Vector Fields}

\author[W. Huang, X. Zhang
]
{Wenyong Huang$^{a}$, Xiang Zhang$^b$
 }

\address{$^a$School of Mathematical Sciences, Shanghai Jiao Tong University, Shanghai 200240, People's Republic of China}
\email{sjtuhwy123@sjtu.edu.cn}

\address{$^b$School of Mathematical Sciences, MOE--LSC,  and CAM--Shanghai, Shanghai Jiao Tong University, Shanghai 200240, People's Republic of China} \email{xzhang@sjtu.edu.cn}

	
\begin{document}

\begin{abstract}
Prelle and Singer showed in 1983 that if a system of ordinary differential equations defined on a differential field $K$ has a first integral in an elementrary field extension $L$ of $K$, then it must have a first integral consisting of algebraic elements over $K$ via their constant powers and logarithms. Based on this result they further proved that an elementary integrable planar polynomial differential system has an integrating factor which is a fractional power of a rational function. Here we extend their results and prove that any $n$  dimensional elementary integrable polynomial vector field has $n-1$ functionally independent first integrals being composed of algebraic elements over $K$. {Furthermore, using the Galois theory we prove that the vector field has a rational Jacobian multiplier.}
\end{abstract}

\maketitle

\section{Introduction and statement of the main results}

Consider the autonomous rational vector field
\begin{equation}\label{vector field}
\mathcal{X}:={P_1}(x)\frac{\partial }{{\partial {x_1}}} +  \cdots  + {P_n}(x)\frac{\partial }{{\partial {x_n}}},
\end{equation}
where $x=(x_1,\ldots,x_n)\in\mathbb C^n$, and $P_1,\ldots,P_n \in K:=\mathbb{C}(x)$, the field of rational functions in $x$.
Here we concern with reduction of elementary integrability of the vector field $\mathcal X$. From this point of view it is not necessary to distinguish whether the vector field $\mathcal X$ is rational or polynomial. Since we are working in the elementary functions over $K=\mathbb C(x)$, we first recall some related terminologies on elementary field extensions.

A \textit{derivation} on a ring $\mathcal{R}$ is an operator $\delta : \mathcal{R} \longrightarrow \mathcal{R}$ which satisfies
\[\delta (f+g)=\delta f+\delta g,\quad \delta(fg)=(\delta f)g+f(\delta g),\quad \text{for all} \ f,g \in \mathcal{R}.\]
A \textit{differential field} $(\mathcal{F},\Delta)$ consists of the field $\mathcal{F}$ and the set $\Delta$ of commutative derivations defined on $\mathcal{F}$.
A differential field $(L,\Delta)$ is an {\it elementary field extension} of the differential field $(K,\Delta)$ with  $\Delta:=\left\{\partial_{x_1}= {\partial}/{\partial x_1}, \ldots,\partial_{x_n}= {\partial}/{\partial x_n}\right\}$ a set of the commutative derivations on $K=\mathbb C(x)$, if this extension can be written in the tower form
\[
K=K_0\subset K_1 \subset \cdots \subset K_r=L
\]
such that each field extension satisfies one of the following three properties
\begin{itemize}
    \item $K_{i+1}$ is a finite algebraic extension of $K_i$; or

    \item $K_{i+1}=K_i(t)$, where $t$ satisfies that there exists an $h \in K_i$ such that $\frac{\delta t}{t}=\delta h$ for all $\delta \in \Delta$; or

    \item $K_{i+1}=K_i(t)$, where $t$ satisfies that there exists an $h \in K_i$ such that $\delta t = \frac{\delta h}{h}$  for all $\delta \in \Delta$.
\end{itemize}
An elementary field extension is actually made up of a finite number of smallest field extensions, each of which adds either a finite number of algebraic elements, or an exponential element, or a logarithmic element to the former field. Each $K_i$, $i=1,\ldots,r$, is called a \textit{tower element} of the elementary field extension. The elementary functions of a single variable were introduced by Liouville from $1833$ to $1843$ in his study of the integration of functions, see \cite[Section 3.3]{Zhang2017} for details.
In what follows, denote by $C(K,\Delta)$ the constants field, i.e. the set of elements in $K$ that vanish under all $\delta\in\Delta$. This notation is also similarly applied to $C(L,\Delta)$.

For the rational vector field $\mathcal X$, the smooth functions $H_1,\ldots,H_k$ defined in $\Omega=\mathbb C^n\setminus C_0$, with $C_0$ a zero Lebesgue measure subset of $\mathbb C^n$, are functionally independent first integrals of $\mathcal X$ if $\mathcal X(H_j)\equiv 0$ in $\Omega$, $j=1,\ldots,k$, and their gradients, i.e. the $k$ vector-valued function
\[
\nabla_xH_1,\ldots,\nabla_xH_k,
\]
are linearly independent in $\Omega$ except perhaps a zero Lebesgue measure subset. If $H_j$ is in an elementary field extension of $(K,\Delta)$ and is a first integral of $\mathcal X$, it is called an {\it elementary first integral} of $\mathcal X$.  The vector field $\mathcal X$ is {\it elementary integrable} if it has $n-1$ functionally independent elementary first integrals.
Recall that a {\it Jacobian multiplier} of the vector field $\mathcal X$ is a nonconstant  smooth function $J(x)$ that satisfies $\mathcal X(J)=-J\mbox{\rm div}\mathcal X$, where
$\mbox{\rm div}\mathcal X$ represents the divergence of the vector field $\mathcal X$. Specially, when $n=2$ a Jacobian multiplier is called an {\it integrating factor}.

Prelle and Singer \cite{Prelle-Singer} in 1983 provided a reduction on elementary first integrals, as stated below, which is the main result of their paper.

\noindent{\bf Theorem A.} {\it Let $(L,\Delta)$ be an elementary field extension of the differential field $(K,\Delta)$ with $C(L,\Delta)=C(K,\Delta)$. Assume that $C(L,\Delta)$ is a proper subset of $C(L,\mathcal X)$, that is,  $\mathcal X$ has a nontrivial elementary first integral. Then there exist elements of $L$, $w_0,w_1,\ldots,w_m$, which are algebraic over $K$, and $c_1,\ldots,c_m$ in $C(K,\Delta)$ such that
	\[
	\mathcal X(w_0)+\sum\limits_{i=1}\limits^mc_i\frac{\mathcal X(w_i)}{w_i}=0 \quad \mbox{and} \quad
	\delta(w_0)+\sum\limits_{i=1}\limits^mc_i\frac{\delta(w_i)}{w_i}\ne 0\quad \mbox{for some } \ \delta\in \Delta.
	\]
}

Restricted to a planar polynomial vector field $\mathcal P$, Prelle and Singer \cite{Prelle-Singer} further proved in their Propositions 1 and 2 the next results.

\noindent{\bf Proposition A.} {\it If a planar polynomial vector field $\mathcal P$ has an elementary first integral, then the following statements hold.
\begin{itemize}
\item[$(a)$] The vector field $\mathcal P$ has an algebraic integrating factor. That is, there exists an $R\ne 0$ algebraic over $K$ such that $\mathcal P(R)=-R\mbox{\rm div} \mathcal P	$.
\item[$(b)$] The vector field $\mathcal P$ has either a nontrivial rational first integral, or an algebraic integrating factor $R$ satisfying $R^r\in K$ for some $r\in \mathbb Z$.
\end{itemize}
}
Using the language given in \cite[p.33]{CLY2014}, we can simply say that if a planar polynomial vector field $\mathcal P$ has an elementary first integral, then it has an integrating factor which is a fractional power of a rational function. In 1997, Man and MacCallum \cite{Man-MacCallum} attempted to implement Prelle and Singer's result through symbolic integration. Avellar et al \cite{ADDdaM2007} in 2007 also provided some algorithm to find elementary first integrals of rational second order ordinary differential equations.
In 2019 Christopher et al \cite{CLPW2019} focused on the conditions under which a planar polynomial differential system having an elementary first integral admits also a Darboux first integral among others.

The result on reduction of elementary first integrals was extended to the Liouvillian class by Singer \cite{Singer1992} in 1992 for planar polynomial vector fields. Recall that a {\it Liouvillian field extension} of a differential field is an extension  which can be written in the tower form, as in the elementary field extension, only replacing the transcendental element $t$ in the third item by integration of an element in $K_i$, that is, $K_{i+1}=K_i(t)$ with $t$ satisfying $\delta t\in K_i$ for each $\delta\in \Delta$, see \cite[Section 3.3]{Zhang2016} for details.  The Singer's result can be stated as follows.

\noindent{\bf Theorem B. }{\it Assume that the polynomial differential system
\[
\dot x=P(x,y),\qquad \dot y=Q(x,y),
\]
with $P$, $Q \in \mathbb C[x,y]$, has an analytic solution $(x,y)=(\phi (t),\varphi(t))$ defined in an open subset $V\subset \mathbb C$. If there is a Liouvillian function $F(x,y)$ which is analytic in an open subset containing $\mathcal{O}:=\{(\phi (t),\varphi(t))|\ t\in V\}$, and $F(x,y)|_{\mathcal{O}}=0$, then either $\mathcal{O}$ is an algebraic solution, or the system has an integrating factor of the form
\[
R(x,y):=\left(\int_{(x_0,y_0)}^{(x,y)}U(x,y)dx+V(x,y)dy\right),
\]
where $U$, $V \in \mathbb C(x,y)$ satisfy $\partial_yU=\partial_xV$.}

This theorem shows that a planar Liouvillian integrable polynomial differential system has an integrating factor which can be expressed in an exponential of a rational closed one-form.
Based on the Singer's result, Christopher \cite{Christopher1999} further proved that if a planar polynomial differential system is Liouvillian integrable, it has a Darboux integrating factor. On the Christopher's result, there was a different approach obtained by Zoladek \cite{Zoladek1998} in terms of the monodromy group of the first integral. In 2016, Zhang \cite{Zhang2016} generalized the results of Singer \cite{Singer1992} and Christopher \cite{Christopher1999} to $n$--dimensional polynomial differential systems and provided a reduction via Darboux Jacobian multipliers when the system is Liouvillian integrable.
For more information on elementary and Liouvillian first integrals and their reductions, see for example, the books \cite{CLY2014,Zhang2017} and the articles \cite{B.Jamil,Llibre2003,Duarte,J.Gine,S.Varad}.

The reductions in \cite{Singer1992,Christopher1999} for two dimensional systems and in \cite{Zhang2016} for any finite dimensional systems provide a relatively easier  tool to characterize Liouvillian integrability of a given polynomial differential system. Now, the question is whether the reduction in \cite{Prelle-Singer} on two dimensional polynomial vector fields having an elementary first integral to any finite dimensional polynomial vector fields which are elementary integrable.  

In this paper we present a positive answer to the above question, which consists of the next two theorems, i.e. Theorems \ref{Theorem1} and \ref{Theorem2}.
Our first main result is an extension of Theorem A, which is one of the essential components to address the question.

To state our results, denote by $\overline{F}$ the algebraic closure of a field $F$, which by definition is the field with all the algebraic elements in $F$.
On existence of algebraic closures of a given field, see e.g. \cite[Section 3.4, Theorem II]{Hodge}.

\begin{thm}\label{Theorem1}
Let $(L,\Delta)$ be an elementary field extension of the differential field $(K,\Delta)$, with $C(L,\Delta)=C(K,\Delta)$. Assume that $C(L,\Delta)\subsetneqq C(L,\mathcal{X})$, that is, the vector field $\mathcal{X}$ defined in \eqref{vector field} has at least a nontrivial elementary first integral.
If the vector field $\mathcal{X}$ has $k~(1\leq k \leq n-1)$ functionally independent first integrals $H_1,\ldots,H_k$ in $L$, then there exist $k$ functionally independent first integrals $\widetilde{H}_1,\ldots,\widetilde{H}_k$, which are of the form
\[\widetilde{H}_j = {u_{0,j}} + \sum\limits_{i = 1}^{{n_j}} {{c_{i,j}}\ln {u_{i,j}}},\]
where $u_{0,j},\ldots ,u_{n_j,j}\in \overline{K}\cap L$, $c_{1,j},\ldots,c_{n_j,j}\in \mathbb C$, $j=1,\ldots,k$.
\end{thm}

We remark that when $k=1$, Theorem \ref{Theorem1} is exactly the main Theorem of Prelle and Singer \cite{Prelle-Singer}. Here we extend their result to any $k\in \{1,\ldots,n-1\}$ number of functionally independent elementary first integrals.
Theorem \ref{Theorem1} ensures existence of a suitable number of functionally independent first integrals consisting of algebraic elements over $K$ for the vector field $\mathcal X$ provided that it has the corresponding number of functionally independent elementary first integrals.

{The difficulty in proving Theorem \ref{Theorem1} is to find functionally independent first integrals of the given simpler form in the elementary field extension. According to Prelle and Singer's work, when the vector field  \eqref{vector field} possesses an elementary first integral $H$, it has a first integral of the form $\widetilde H=w_0+{\rm ln}w_1$ by employing \cite[Proposition 1.2]{RischRObert}. However, the lack of a relationship between  $H$ and $\widetilde H$ is a primary obstacle in proving Theorem \ref{Theorem1}. Specifically, let $H_1$ and $H_2$ be functionally independent elementary first integrals of the vector field $\mathcal X$. Applying \cite[Proposition 1.2]{RischRObert} to $H_1$ and $H_2$ respectively, we can determine $\widetilde H_1$ and $\widetilde H_2$ as first integrals in the forms akin to those stated in Theorem \ref{Theorem1}. Nevertheless, we cannot guarantee whether $\widetilde H_1$ and $\widetilde H_2$ are functionally independent, or even  $\widetilde H_1 \ne \widetilde H_2$.
Therefore, it is imperative to meticulously establish their relationship. We note that the independence of functions pertains to analytic aspects, whereas the algebraic independence does not bear on it. For instance, $x+y$ and $e^{x+y}$ are functionally dependent, yet they are algebraically independent on the field $\mathbb C(x,y)$. Moreover, any two functionally independent rational first integrals of system \eqref{vector field} are algebraically dependent on the field $K$.

Based on Theorem \ref{Theorem1}, we can get a reduction via Jacobian multipliers on any finite dimensional elementary integrable rational vector field $\mathcal X$, defined in \eqref{vector field}, which is a generalization of Propositions 1 and 2  obtained by Prelle and Singer \cite{Prelle-Singer} from two dimension to any finite dimension.

\begin{thm}\label{Theorem2}
If the vector field $\mathcal X$ defined in \eqref{vector field} is elementary integrable, i.e. it has $n-1$ functionally independent elementary first integrals, then the vector field $\mathcal X$ has a rational Jacobian multiplier in $\mathbb{C}(x)$.
\end{thm}

We remark that applying Theorem \ref{Theorem2} to concrete systems, one can get some
necessary conditions on elementary integrability of the systems. For example, the Lorenz system \cite{Lorenz1963}
\begin{equation}\label{eLorenz}
	\dot x=s(y-x),\quad \dot y=rx-y-xz,\quad \dot z=-b z+x y,
\end{equation}
with $s,r,b$ parameters, has the six irreducible invariant algebraic surfaces (Darboux polynomials in terms of the terminology from the Darboux theory of integrability):
\[
\begin{array}{ccc}
  {\rm Darboux \ polynomial}  & {\rm Cofactor}  &  {\rm Parameter}\\
 	x^2-2s z     &  -2s  &  b=2s\\
   x^4-\frac 43 x^2z-\frac 49 y^2-\frac 89 xy+\frac 43 rx^2 & -\frac 43 & b=0, s=\frac 13\\
  y^2+z^2    &  -2  & b=1, r=0\\
	\end{array}
\]
found by Segur \cite{Segur1982} in 1982, and
\[
\begin{array}{ccc}
  {\rm Darboux \ polynomial}  & {\rm Cofactor}  &  {\rm Parameter}\\
  x^4-4x^2 z-4 y^2+8xy-4r x^2-16(1-r)z &  -4   & b=4,s=1\\
  y^2+z^2-rx^2   & -2  & b=1, s=1\\
  x^4-4sx^2z-4s^2y^2+4s(4s-2)xy-(4s-2)^2x^2  & -4s  & b=6s-2, r=2s-1
\end{array}
\]
obtained by K\'us \cite{Kus1983} in 1983.  Llibre and Zhang \cite{LlibreZhang2002} in 2002 completed the classification and proved that the six classes are the only irreducible ones. On this classification, Swinnerton-Dyer \cite{Swinnerton-Dyer} provided a different approach. Recall that a polynomial $f\in\mathbb C[x,y,z]$ is a {\it Darboux polynomial} of the vector field $\mathcal L$ associated to system \eqref{eLorenz} if $\mathcal L(f)=kf$ with $k\in \mathbb C[x,y,z]$, called a {\it cofactor} of $f$.  Correspondingly, $\{(x,y,z)\in\mathbb R^3|\ f(x,y,z)=0\}$ is an {\it invariant algebraic surface}.
According to the Darboux theory of integrability (see e.g. Llibre \cite[Theorem 2.1]{Llibre2004}, Zhang \cite[Proposition 3.5]{Zhang2017}), any rational Jacobian multiplier is of the form $J=f_1^{k_1}\cdot\ldots\cdot f_p^{k_p}$ with $f_j$'s Darboux polynomials and $ k_j\in\mathbb Z$. Moreover,  $J=f_1^{\ell_1}\cdot\ldots\cdot f_p^{\ell_p}$ is a Jacobian multiplier of $\mathcal L$ if and only if $\ell_1 k_1+\ldots\ell_pk_p=-\mbox{\rm div}\mathcal L$, where $k_j$ is the cofactor associated with the Darboux polynomial $f_j$, $j=1,\ldots,p$.
By verifying this last equality using the cofactors given above one can get the necessary and sufficient conditions on the existence of the rational Jacobian multipliers, and consequently the necessary conditions on elementary integrability of system \eqref{eLorenz}.

In the proof of Theorem \ref{Theorem1} we need factorization of polynomials in $K_i(\lambda)$ with $\lambda$ a transcendental element over $K_i$, for instance $\lambda=e^{u}$ or $\lambda=\ln u$ with $u$ in $K_i$, an element in the tower.} Let $P(\lambda)$ and $R(\lambda)$ be polynomials in $\lambda$ with coefficients in $K_i$. If we view $P$ and $R$ as elements in the ring $K_i[\lambda]$ (or the quotient field $K_i(\lambda)$), then by Hodge and Pedoe \cite[Section 3.2]{Hodge} we can explore their algebraic relationships, such as coprimality, division and factorization. Additionally, since they are functions in $x$, we can examine their functional independence and smoothness. In Section \ref{Pre}, we will introduce essential concepts integral to our proofs.

The remaining part of this paper proceeds as follows. Section \ref{Proof of Theorem1} focuses on reduction of functionally independent first integrals, presenting the proof of Theorem \ref{Theorem1}. Based on the results in Section \ref{Proof of Theorem1}, we provide a proof in Section \ref{S3} on reducing the Jacobian multipliers of the elementary integrable vector field $\mathcal X$ that has the functionally independent first integrals being composed of algebraic elements over $K$ by using {the Galois group of the field extensions.} 

\section{Preliminaries on algebra and analysis}\label{Pre}
Let $F$ be an extension field of a field $E$, denoted by $F/E$.
$\alpha \in F$ is called an \textit{algebraic element} over $E$ if there exists a polynomial $P$ with coefficients in $E$ such that $P(\alpha)=0$. Otherwise, we say that $\alpha$ is a {\it transcendental element} over $E$. If all elements of $F$ are algebraic over $E$, we call $F/E$  an {\it algebraic field extension} of $E$ or simply say that $F$ is {\it  algebraic over $E$}.

Let $\xi$ be an element of $F$. Then, the ring $E[\xi]:=\{\alpha_0+\alpha_1\xi+\ldots+\alpha_r\xi^r|\ \alpha_i\in E, 0\leq r\in \mathbb N\}$ is a commutative integral domain. Nothing that $E[\xi]=E$ if and only if $\xi \in E$. We denote by $E(\xi)$ the quotient field of $E(\xi)$, and refer to it as {\it the field obtained by adjoining $\xi $ to $E$}. Furthermore, if $\xi $ is algebraic over $E$, then $E(\xi)=E[\xi]$. Otherwise, $E[\xi]$ is equivalent to an integral domain $E[X]$, with $X$ an indeterminate. Therefore, $E(\xi)$ is equivalent to the field of rational functions in the indeterminate $X$ over $E$. In this situation, the field $E(\xi)$ is called the {\it transcendental extension} of $E$. Obviously, $E(\xi)$ is the smallest extension of $E$ which contains $\xi$. Similarly, we can define the field $E(\xi_1,\ldots,\xi_n)$ (some of $\xi_i$'s may be algebraic over $E$) inductively as follows: $E(\xi_1,\ldots,\xi_r)$ is the field obtained by adjoining $\xi_r$ to $E(\xi_1,\ldots,\xi_{r-1})$. The statements referred above can be found in \cite[Section 3.2]{Hodge}.

Denote by $E[X_1,\ldots,X_k]$ the ring of polynomials in the variables  $X_1,\ldots,X_k$ with coefficients in $E$. Elements $y_1,\ldots,y_k \in F$ are called \textit{algebraically dependent} over $E$ if there exists a nontrivial $f \in E[X_1,\ldots,X_k]$ such that $f(y_1,\ldots,y_k)=0$. Otherwise, we call $y_1,\ldots,y_k$ \textit{algebraically independent}, that is, any polynomial $g\in E[X_1,\ldots,X_k]$ satisfying $g(y_1,\ldots,y_k)=0$ is trivial, i.e., a zero polynomial. A nonalgebraic element of $E$ is {\it transcendental} over $E$.

Let $S=\{y_\alpha:\alpha \in I\}$, with $I$ an index set, be a subset of $F$ which is algebraically independent over $E$. If the cardinality of $S$ is the greatest one among all such subsets, then we call this cardinality the \textit{transcendence degree} of $F$ over $E$, and write the cardinality as ${\rm tr}[F:E]$. As a result, ${\rm tr}[F:E]=0$ means that $F$ is algebraic over $E$.
Furthermore, refer to \cite[Chapter VIII, Exercise 3]{Lang}, we have
\begin{lem}\label{tr gongshi}
Let $U \subset E \subset F$ be extension fields. Then
\[
{\rm tr}[F:U]={\rm tr}[F:E]+{\rm tr}[E:U].
\]
\end{lem}
\noindent Moreover, if $S$ is maximal with respect to the inclusion ordering, we call $S$ a \textit{transcendence base} of $F$ over $E$. From the maximality, it is clear that if $S$ is a transcendence base of $F$ over $E$, then $F$ is algebraic over $E(S)$, see \cite[Chapter VIII]{Lang} for details.

Considering $\theta$ as a transcendental element over $E$, combining the theory of field extension (see \cite[Section 3.2]{Hodge}), it follows the next result.
\begin{prop}\label{extension gonghsi}
Set $\gamma=a+b{\theta^m}$, where $a,b\in E$ and $b\ne 0$, $m$ is non-zero integer. If $\theta$ is transcendental over $E$, then $E(\gamma) =E(\theta)$ if and only if $m =\pm 1$.
\end{prop}
\begin{proof}
It is obvious that $E[\gamma] \subseteq E[\theta]$. We consider only the case $m>0$.
{\it Sufficiency}: If $m=1$, $\theta=(\gamma-a)/b\in E[\gamma]$, and so $E[\gamma] = E[\theta]$.
{\it Necessity}:  If $E(\gamma) =E(\theta)$, we write $\theta=M/N$, with $M,N\in E[\gamma]$. Set $\deg M=s\geq 1$ and $\deg N=r\geq 0$. Regarding $M$ as a polynomial in $E[\theta]$, then $\deg M=sm$ and $\deg N=rm$.
The equality $N\theta=M$ forces that $rm+1=sm$. Note that $r,s,m$ are positive integers, it follows that $m=1$.
\end{proof}

Let $A$ be a subring of a commutative ring $B$ and $\beta \in B$.
According to \cite[Section 2.3]{Lang}, if $\beta$ is transcendental over $A$, then $A[\beta]$ is isomorphic to $A[z]$, with $z$ an indeterminate. Consequently, the ring $A[\beta]$ is factorial and principal.

Based on these facts, we consider the field $K=\mathbb C(x_1,\ldots,x_n)$. Let $f \in K[\theta]$ be monic and of degree $r$, i.e., $f=\theta^r+a_{r-1}\theta^{r-1}+ \cdots+a_0$, with $a_i\in K$, $i=1,\ldots,r-1$, and $\deg f=r\geq 1$. Consider a derivation $D$ such that  $(K,D)\subset(K(\theta),D)$ is a differential field extension.

\noindent{Case 1.} $\theta={\rm ln}v$ for some non-constant $v\in K$. Then, $D(f)=(Da_{r-1}+r(Dv/v))\theta^{r-1}+\cdots+(Da_0+a_1(Dv/v))$. Note that $\theta$ is transcendental over $K$ and $Df\in K[\theta]$. Hence, if $D(f)=0$, we have $Da_{r-1}+r(Dv/v)=\cdots=Da_0+a_1(Dv/v)=0$. Otherwise, $\theta$ is algebraic over $K$, which is a contradiction.

\noindent{Case 2.} $\theta=e^v$ for some non-constant $v \in K$. Then $Df=rD(v)\theta^r+(Da_{r-1}+(r-1)D(v))\theta^{r-1}+\cdots+(Da_1+Dv)\theta+Da_0$. Hence, $Df=0$ forces that $rD(v)=Da_{r-1}+(r-1)D(v)=\cdots=Da_0=0$, which means that $0=Dv=Da_0=\cdots=Da_{r-1}$.
\\
See \cite{RischRObert} for similar discussions.



\section{Proof of Theorem \ref{Theorem1}: Reduction of elementary first integrals}\label{Proof of Theorem1}

In this section we establish a reduction on functionally independent elementary first integrals over $K=\mathbb C(x)$ to functionally independent first integrals in simpler forms, which are in fact composed of a finite $\mathbb C$--linear summation of  algebraic elements over $K$ and logarithms of algebraic elements.

The next result, due to Bruns \cite{Bruns1887}, provides a reduction of algebraic integrability to rational integrability. For a proof, see e.g. Forsyth \cite[pp.323-335]{Forsyth1900} or Zhang \cite[Proposition 5.2]{Zhang2017}.

\begin{lem}\label{algebraic lemma}
Let $(F,\mathcal Y)$ be a differential field, and $(L,\mathcal Y)$ be an algebraic extension of $(F,\mathcal Y)$.
 If $H \in {L\setminus F}$ satisfies $\mathcal Y (H)=0$, then all the coefficients of the characteristic polynomial for $H$ belong to $C(F,\mathcal Y)$.
\end{lem}

\begin{proof} Since its proof is short, for completeness we present it here. By the assumption, let
 \[
 P(X) = {X^r} + {a_{r - 1}}{X^{r - 1}} +  \cdots  + a_1 X+ {a_0}
 \]
 be the characteristic polynomial of $H$, where $a_1,\ldots,a_{r-1} \in F$. Then $P(H)=0$.
 Applying the derivation $\mathcal{Y}$ on the both sides of the equality $P(H)=0$ yields
 \[
   \mathcal Y \left(a_{r - 1} \right){H^{r - 1}} +  \cdots +  \mathcal Y \left(a_1\right) H +  {\mathcal Y (a_0)} =0.
 \]
 By minimality of the degree of the characteristic polynomial $P(X)$, it follows from the last equality that $\mathcal Y(a_j) =0$, $j=1,\ldots,r-1$.
That is, $ a_j  \in C(F,\mathcal{Y})$.
\end{proof}

Instead of proving Theorem \ref{Theorem1} directly, we prove a seemly stronger statement as follows. The main idea treating in this way is from \cite{Prelle-Singer}.

\begin{thm}\label{Thm Stronger}
Let $(K,\Delta)$, $(L,\Delta)$ and $\mathcal{X}$ be as those in Theorem \ref{Theorem1}. Assume that there exist $u_{0,j},\ldots,u_{m_j,j} \in L$, and $d_{1,j},\ldots,d_{m_j,j} \in \mathbb{C}$, $j=1,\ldots,k$, such that
\begin{itemize}
    \item $\mathcal{X}(H_j)=0$, $j=1,\ldots,k$,
        where $H_j={u_{0,j}} + \sum\limits_{i = 1}^{{m_j}} {{d_{i,j}}\ln {u_{i,j}}}$;

    \item $\delta (H_j)\ne 0$ for some $\delta \in \Delta$, $j=1,\ldots,k$;

    \item $H_1,\ldots, H_k$ are functionally independent.
\end{itemize}
Then there exist $w_{0,j},\ldots,w_{\widetilde m_j,j} \in L$, $j=1,\ldots,k$, which are algebraic over $K$, and $c_{0,j},\ldots,c_{\widetilde m_j,j} \in \mathbb C$, such that 
\begin{itemize}
 \item $\mathcal{X}(\widetilde H_j)=0$, $j=1,\ldots,k$,
        where $\widetilde H_j={w_{0,j}} + \sum\limits_{i= 1}^{{\widetilde m_j}} {{c_{i,j}}\ln {w_{i,j}}}$;

    \item $\delta (\widetilde H_j)\ne 0$ for some $\delta \in \Delta$, $j=1,\ldots,k$;

    \item $\widetilde H_1,\ldots, \widetilde H_k$ are functionally independent.
\end{itemize}
\end{thm}

We remark that the hypotheses of Theorem \ref{Thm Stronger} certainly hold if the hypotheses of Theorem \ref{Theorem1} are satisfied. Conversely, if the assumptions of Theorem \ref{Thm Stronger} hold, then, {as did in \cite{Prelle-Singer} by Prelle and Singer}, in some elementary extension of $L$, $\widetilde{H}_j = {u_{0,j}} + \sum\limits_{j = 1}^{{\tilde m_j}} {{c_{i,j}}\ln {u_{i,j}}}$, $j=1,\ldots,k$, satisfies $\mathcal{X}(\widetilde{H}_j)=0$ and  $\delta(\widetilde{H}_j)\ne 0$ for some $\delta\in \Delta$. Moreover, according to Theorem \ref{Thm Stronger},  $\widetilde{H}_1,\ldots,\widetilde{H}_k$ are functionally independent.

We will prove Theorem \ref{Thm Stronger} by induction on the transcendental degree tr$[L:K]$ of the elementary field extension $L/K$ replacing the order of the tower in the definition of the field extension, because tr$[L:K]$ is more essential than the order in the proof of Theorem \ref{Thm Stronger}, see the treatments adopted in \cite{M.Rosenlicht,RischRObert,Prelle-Singer}. In fact, if we prove inductively on the order $s$ of the elementary field extension $L/K$, regardless of the value of $s$, it may happen the case ${\rm tr}[K_{s+1}:K_s]=0$, which means that $L$ is algebraic over $K$. By the induction on ${\rm tr}[K_{i+1}:K_i]$, the theorem holds. But if our induction is on the order of the tower, there maybe have several similar discussions on the case  ${\rm tr}[K_{i+1}:K_i]=0$. In this sense, it will be more tedious to consider induction on the order of the tower of the elementary field extension $L/K$ than on  tr$[L:K]$.

We first assume that $L$ is an algebraic extension of $K(t)$, where $t$ is an transcendental element over $K$ satisfying either $ {\delta t}/{t} =\delta v $ for some $v\ne 0$ in $K$ and all $\delta$ in $\Delta$, i.e. $t=e^{v}$, or $\delta t$=$\delta v$/$v$ for some $v \ne 0$ in $K$ and all $\delta$ in $\Delta$, i.e. $t={\rm ln}v$.  Consider the tower of the field extension $K\subset K(t) \subset L$.
We distinguish the number $k$ of the functionally independent first integrals of the vector field $\mathcal X$.

\noindent\textit{Case $1 :k=1$}. The result is \cite[Theorem]{Prelle-Singer}, i.e. Theorem A stated above. So one of the proofs was given there.

\noindent\textit{Case $2 :k>1$}.
The next result, due to Rosenlicht \cite[Theorem 3]{M.Rosenlicht}, will be helpful to our next proofs.
\begin{thm}\label{Thm M.Rosenlicht}
Let $\widetilde{k}$ be a differential field of characteristic zero, and
for each given derivation $D$ of $\widetilde{k}$ let $\alpha_D \in \widetilde{k}$. Then there exists an elementary differential field extension of $\widetilde{k}$ having the same constants field as $\widetilde k$ and containing an element $y$ such that $Dy=\alpha_D$ {for each given derivation $D$}
if and only if there are constants $c_1,\ldots,c_\ell \in \widetilde{k}$ and elements $u_1,\ldots,u_\ell,v\in \widetilde{k}$, such that { for each given derivation $D$} we have
 \[
 {\alpha _D} = Dv + \sum\limits_{i = 1}^\ell {c_i}\frac{{D{u_i}}}{{{u_i}}}.
 \]
\end{thm}

Christopher et al \cite[Proposition 3]{CLPW2019} provided an application of Theorem \ref{Thm M.Rosenlicht} to planar polynomial differential systems which have an elementary first integral. Based on Theorem \ref{Thm M.Rosenlicht}, we can prove the next result.

\begin{prop}\label{Prop1}
Assume that  the functionally independent first integrals $H_1,\ldots,H_k$ of $\mathcal X$ satisfy $\delta H_j \in K$, $j=1,\ldots,k$, for all $\delta \in \Delta=\{\partial _{{x_1}}, \cdots ,\partial _{{x_n}}\}$. Then there exist $\widetilde{H}_j$, $j=1,\ldots,k$, being of the form
\[
\widetilde{H}_j=w_{0,j} + \sum\limits_{i = 1}^{{m_j}} {{c_{i,j}}\ln {w_{i,j}}},\ \
 w_{0,j}, w_{1,j},\ldots,w_{m_j,j} \in K, \ \  c_{1,j},\ldots,c_{m_j,j} \in \mathbb C ,
 \]
which are functionally independent first integrals of $\mathcal X$.
\end{prop}

\begin{proof}
Applying Theorem \ref{Thm M.Rosenlicht} to $H_j$, $j \in \{1,\ldots,k\}$, yields that there are $w_{0,j},\ldots,w_{m_j,j} \in K$ and $c_{1,j},\ldots,c_{m_j,j} \in \mathbb C$, $m_1,\ldots,m_k \in \mathbb N$, such that, for all $\delta \in \Delta$,
\[
\delta H_j=\delta{w_{0,j}} + \sum\limits_{i = 1}^{{m_j}} {{c_{i,j}}\frac{\delta{w_{i,j}}}{{{w_{i,j}}}}},\quad j=1,\ldots,k.
\]
Set $\widetilde{H}_j=w_{0j} + \sum\limits_{j = 1}^{{m_j}} {{c_{ij}}\ln {w_{ij}}}$, $j=1,\ldots,k$. Then 
$\delta H_j$=$\delta \widetilde{H}_j$, $j=1,\ldots,k$, which means that 
$\mathcal{X}(\widetilde H_j)=\mathcal{X}(H_j)=0$, and $\widetilde H_1,\ldots, \widetilde H_k$ are functionally independent due to those  of $H_1,\ldots,H_k$, the results that we want to prove.%
\end{proof}

Next we turn to the proof of Theorem \ref{Thm Stronger} in the case $\delta H_j\notin K$ for some $\delta\in \Delta$. The proof will be processed via induction on the number of functionally independent first integrals.
We firstly consider the case $k=2$.

\begin{prop}\label{Prop2}
Consider the tower of the differential field extension $K\subset K(t) \subset L$ with $t$ an transcendental element over $K$ and $L$ an algebraic extension of $K(t)$. Let $H_1$ and $H_2$ be two functionally independent first integrals of $\mathcal{X}$ in $L$.
If $\mathcal{X}(t)=0$, then Theorem \ref{Thm Stronger}
holds.
\end{prop}
\begin{proof} By the assumption of the proposition, let
\begin{equation}\label{minimal Eq}
\begin{split}
H_1^\ell + {a_{\ell - 1}}H_1^{\ell - 1} +  \cdots  + {a_1}{H_1} + {a_0}& = 0,\\
H_2^m + {b_{m - 1}}H_2^{m - 1} +  \cdots  + {b_1}{H_2} + {b_0}& = 0
\end{split}
\end{equation}
be the minimal algebraic equations over $K(t)$ that $H_1$ and $H_2$ satisfy, respectively,
where $a_j$, $b_i \in K(t)$, $j=0, 1,\ldots,\ell-1$, $i=0, 1,\ldots,m-1$. By Lemma \ref{algebraic lemma}, we know that the set $S_1:=\{a_0,\ldots,a_{n-1},b_0,\ldots,b_{m-1}\}\subset C(K(t),\mathcal{X})$. We claim that there exist at least two elements in $S_1$, saying $a$ and $b$, which are functionally independent, {i.e. ${\rm d}a \wedge {\rm d}b \ne 0$, where d represents the exterior differentiation. Indeed, by contrary we assume that for any two elements $a$, $b \in S_1$, one has ${\rm d}a \wedge {\rm d}b =0$.  For any $a\in S_1$, by taking exterior differentiation on \eqref{minimal Eq} and doing wedge product with d$a$, it follows that
\begin{equation}\label{minimal Eq2}
\begin{split}
(l{H_1}^{l-1}+{(l-1)}a_{l-1}{H_1}^{l-2}+\cdots+a_1){\rm d}H_1\wedge{\rm d}a=&0,  \\
     (m{H_2}^{m-1}+{(m-1)}b_{m-1}{H_2}^{m-2}+\cdots+b_1){\rm d}H_2\wedge{\rm d}a=&0.
\end{split}
\end{equation}
By assumption, the degree of the minimal algebraic equations of $H_1$ and $H_2$ are $\ell$, $m$ respectively, hence
\begin{equation*}
\begin{split}
l{H_1}^{l-1}+{(l-1)}a_{l-1}{H_1}^{l-2}+\cdots+a_1\ne& 0,\\
m{H_2}^{m-1}+{(m-1)}b_{m-1}{H_2}^{m-2}+\cdots+b_1\ne& 0.
\end{split}
\end{equation*}
Then, we have ${\rm d}H_1\wedge {\rm d}a={\rm d}H_2\wedge {\rm d}a=0$. Since $H_1$ and $H_2$ are functions uniquely determined by the elements in $S_1$, the arbitrary of $a\in S_1$ forces that ${\rm d}H_1\wedge{\rm d}H_2=0$, a contradiction. The claim follows.}

As a consequence, we have proved that the vector field $\mathcal X$ has two functionally independent first integrals in $K(t)$.

Since $t$ is transcendental over $K$, the ring $K[t]$ is factorial, and so we can write $a$ and $b$ in the form
\begin{equation}\label{factorial}
 a = {h_a}\prod\limits_{j = 1}^r {Q_j^{{r_j}}},\quad
 b = {h_b}\prod\limits_{i = 1}^l {R_i^{{l_i}}},
\end{equation}
where $h_a$, $h_b \in K$, $Q_1,\ldots,Q_r\in K[t]$ (resp. $R_1,\ldots,R_l\in K[t]$) are pairwise distinct monic irreducible polynomials, and $r_j$, $l_i \in \mathbb Z \setminus \{0\}$, $j=1,\ldots,r$, $i=1,\ldots, l$.
Set $S_2 :=\{h_a,Q_1,\ldots,Q_r; h_b,R_1,\ldots,R_l\}$. The similar arguments as utilized in the last paragraph can verify that there exist $Q,R\in S_2$ which are functionally independent, { i.e. ${\rm d}Q \wedge {\rm d}R \ne 0$. If not, since d$a$ (resp. d$b$) can be expressed as a linear combination of d$h_a$, d$Q_1$,\ldots, d$Q_r$ (resp. d$h_b$, d$R_1$,\dots, d$R_l$), it} follows that ${\rm d}a\wedge{\rm d}b=0$, a contradiction.

Substituting the expression \eqref{factorial} of $a$ into the equality $\mathcal{X}(a)=0$ gives
\begin{equation}\label{dengshi1}
{h_a}\prod\limits_{j = 1}^r {Q_j^{{r_j}}} \left( {\frac{{\mathcal{X}({h_a})}}{{{h_a}}} + \sum\limits_{j = 1}^r {{r_j}\frac{{\mathcal{X}({Q_j})}}{{{Q_j}}}} } \right) = 0.
\end{equation}
This equality leads to
\begin{equation} \label{dengshi2}
\frac{{\mathcal{X}({h_a})}}{{{h_a}}}\prod\limits_{j = 1}^r {{Q_j}}  + \sum\limits_{j = 1}^r {{r_j}} {Q_1} \cdots \mathcal{X}({Q_j}) \cdots {Q_r} = 0.
\end{equation}
Note that the left hand side of the equality \eqref{dengshi2} is a polynomial in $t$ with coefficients in $K$, i.e. it belongs to $K[t]$, and that $Q_s$ and $Q_r$ for $s\ne r$ are relatively coprime monic irreducible polynomials. By the expression in \eqref{dengshi2} it follows that $Q_j(t)$ divides $\mathcal X(Q_j(t))$, denoted by $Q_j | \mathcal{X}(Q_j)$, $j=1,\ldots,r$. Applying the same arguments to $b$ yields $R_i | \mathcal{X}(R_i)$, $i=1,\ldots,l$.

Set $q_j:={\rm deg}(Q_j(t))$ then $q_j\geq1$.  Since $Q_j(t)$ is monic and $\mathcal{X}(t)=0$, it follows that ${\rm deg}(\mathcal{X}(Q_j(t)))\le q_j-1$. The fact $Q_j | \mathcal{X}(Q_j)$ shows that $\mathcal{X}(Q_j(t))=0$.
Based on this fact, the equality \eqref{dengshi1} shows that $\mathcal{X}(h_a)=0$. Similarly, we have $\mathcal{X}(R_i)=0$ and $\mathcal{X}(h_b)=0$. Consequently, $Q$ and $R$ satisfy that $\mathcal{X}\left( Q \right) = \mathcal{X}\left( R \right) = 0$.

If \{$Q$, $R$\}=\{$h_a$, $h_b$\}, this means that $h_a$ and $h_b$ are functionally independent rational first integrals of the vector field $\mathcal{X}$, we are done.

If $\{Q, R\}\ne \{h_a, h_b\}$, without loss of the generality, we assume that ${\rm deg}Q\geq 1$ and
\begin{equation}\label{dengshi3}
\begin{split}
{\text{Q}} = {t^q} + {A_{q - 1}}{t^{q - 1}} &+  \cdots  + {A_1}t + {A_0} \hfill \\
R = {t^p} + {B_{p - 1}}{t^{p - 1}} &+  \cdots + {B_1}t + {B_0} \hfill \\
\end{split}
\end{equation}
where $A_j$, $B_i \in K$, $j=0, 1,\ldots,q-1$, $i=0, 1,\ldots,p-1$. We mention that if ${\rm deg}R=0$ then $B_0\in \{h_a,h_b\}$. Taking the derivation $\mathcal{X}$ on $Q$ yields
\begin{equation}\label{dengshi4}
0 = \mathcal{X}({A_{q - 1}}){t^{q - 1}} +  \cdots  + \mathcal{X}({A_1})t + \mathcal{X}({A_0}).
\end{equation}
Recall that $t$ is transcendental over $K$ and all coefficients of the right hand side of the equality \eqref{dengshi4} are in $K$. These together with \eqref{dengshi4} force that $\mathcal{X}(A_j)=0$, $j=0, 1,\ldots,q-1$. Similarly, we have
$\mathcal{X}(B_i)=0$, $i=0, 1,\ldots,p-1$. Let $S_3:=\{A_0,\ldots,A_{q-1}; B_0,\ldots,B_{p-1}\}$. From independence of $Q$ and $R$, it follows that there exists some $A\in S_3$ such that $A$ and $t$ are functionally independent {(if not, for any $A \in S_3$, one has ${\rm d}A\wedge{\rm d}t=0$, and so ${\rm d}Q\wedge{\rm d}R=0$, a contradiction).} Consequently, if $t={\rm ln}v$, then $A$ and ${\rm ln}v$ are the desired first integrals; if $t=e^{v}$, then $A$ and $v$ are the desired first integrals. The proposition follows.
\end{proof}

Summarizing the proof of Proposition \ref{Prop2} gives the next result.

\begin{Cor}\label{Cor1}
Let $K\subset K(\eta) \subset L$, with $K=\mathbb C(x_1,\ldots,x_n)$, where $\eta$ is transcendental over $K$ and $L$ is an algebraic extension of $K(\eta)$. If $\eta, H_1, H_2 \in C(L,\mathcal{X})$, and $H_1$ and $H_2$ are functionally independent, then either $\mathcal X$ has two functionally independent rational first integrals, or there exists an $\widetilde H_1\in C(K,\mathcal{X})$ such that $\eta$ and $\widetilde H_1$ are two functionally independent first integrals of $\mathcal X$.
\end{Cor}

Up to now, we have proved Theorem \ref{Thm Stronger} for $n=2$ and $K\subset K(t)\subset L$ satisfying $\mathcal X(t)=0$. Next we study the case $\mathcal X(t)\ne 0$.

\begin{prop}\label{Prop3}
Consider the tower of the elementary field extension $K\subset K(t) \subset L$ with $t$ transcendental over $K$ and $L$ algebraic extension of $K(t)$. Assume that $H_1$ and $H_2$ are two functionally independent first integrals of $\mathcal{X}$ in $L$. If $\mathcal{X}(t)\ne0$, then Theorem \ref{Thm Stronger}
holds.
\end{prop}
\begin{proof}
We adopt the same notations as in the proof of Proposition \ref{Prop2}.

\noindent{\it Case} 1. $t={\rm ln}v$. Since $Q_j$, $j=1,\ldots,r$, are monic and $\mathcal{X}(t)=\mathcal{X}(v)/v$, we have for each $Q\in \{Q_1,\ldots,Q_r\}$ via \eqref{dengshi3}
\begin{equation}\label{dengshi3-1}
\begin{split}
\mathcal X(Q) =& \left(\mathcal X(A_{q-1})+q\frac{\mathcal X(v)}{v}\right)t^{q-1}+
\left(\mathcal X(A_{q-2})+(q-1)A_{q-1}\frac{\mathcal X(v)}{v}\right)t^{q-1}\\
&+
\left(\mathcal X(A_{1})+2A_2\frac{\mathcal X(v)}{v}\right)t +\mathcal X(A_0)+A_1\frac{\mathcal X(v)}{v}.
\end{split}
\end{equation}
Now $v\in K$ implies ${\rm deg}\mathcal{X}(Q_j(t))\le q_j-1<q_j={\rm deg}(Q_j(t))$. By $Q_j|\mathcal X(Q_j)$ shows that $\mathcal{X}(Q_j)=0$, {which leads to $\mathcal X(h_a) =0$ via \eqref{dengshi1}}. {Similarly, one has $\mathcal{X}(R_i)=0$, $i=1,\ldots, l$, and $\mathcal X(h_b)=0$.}

If $\{Q,R\}=\{h_a,h_b\}$, we are done. If $\{Q,R\}\ne \{h_a,h_b\}$, we assume without loss of generality that $Q$ is a nonconstant polynomial in $t$. Then  by \eqref{dengshi3-1} gives
\begin{equation}\label{dengshi5}
0 = \mathcal X(Q) = \left( {\mathcal X({A_{q - 1}}) + q\frac{{\mathcal X(v)}}{v}} \right){t^{q - 1}} +  \cdots  + \mathcal X({A_0}) + {A_1}\frac{{\mathcal X(v)}}{v}.
\end{equation}
Since $t$ is transcendental over $K$ and all coefficients of the right hand side of the equality \eqref{dengshi5} are in $K$, these imply that ${\mathcal X({A_{q - 1}}) + q\frac{{X(v)}}{v}}=0$. Furthermore, by the fact that $0\ne \mathcal{X}(t)=\mathcal{X}(v)/v$, we have $\mathcal{X}(A_{q-1})\ne 0$ (otherwise, $\mathcal{X}(v)=0$). Therefore, $A_{q-1} \in K\setminus \mathbb C$, i.e. it is a nonconstant rational function. Now, we take $\widetilde t=A_{q-1}+q{\rm ln}v=A_{q-1}+qt$. Notice that, by Proposition \ref{extension gonghsi}, $K(\widetilde t)=K(t)$, and $\widetilde t$ is transcendental over $K$.  Then $K(\widetilde t)(=K(t))\subset L$ is an algebraic extension by the assumption. Since $\mathcal X(\widetilde t)=0$, it follows from Corollary \ref{Cor1} that there exists an $\widetilde A \in K$ such that $\widetilde A$ is functionally independent of $\widetilde t$ and is a first integral of $\mathcal X$. In summary, $\widetilde A$ and $\widetilde t=A_{q-1}+q{\rm ln}v$ are the desired functionally independent first integrals of $\mathcal X$.

{
\noindent {\it Case }2. $t=e^{v}$. Since $Q_j$ is monic and $\mathcal{X}(t)/t=\mathcal{X}(v)$, by $\mathcal X(t)\ne 0$ gives ${\rm deg}Q_j={\rm deg}\mathcal{X}(Q_j)$.
Therefore, $Q_j|\mathcal{X}(Q_j)$, $j=1,\ldots,r$, implies  $\mathcal{X}(Q_j)= \mathcal{X}(v) q_j Q_j$.

\noindent{\it Subcase} 2.1. $\{ a, b\}=\{h_a,h_b\}$. Recall that $a$, $b$ are functionally independent first integrals of system \eqref{vector field}, and so $h_a$ and $h_b$ are functionally independent rational first integrals. The proof is completed.

\noindent{\it Subcase} 2.2.  $\{ a, b\}\ne \{h_a,h_b\}$. We assume without loss of generality that $a \ne h_a$. Recall that all the $Q_j$s are irreducible polynomials in $K[t]$.

If, for all $j=1,\ldots,r$, $t$ divides $Q_j$, then, $a=h_a t^{\ell}$ for some non-zero integer $\ell$. Note that, since $\mathcal{X}(t)\ne 0$, $h_a \notin C(K,\mathcal{X})$. Therefore, we take $\widetilde t={h_a}e^{\ell v}$. Due to $\widetilde t$ is transcendental over $K$ and, from Proposition \ref{extension gonghsi} and Lemma \ref{tr gongshi}, $K(\widetilde t)(\subseteq K(t))\subset L$ is an algebraic extension by the assumption. Since $\mathcal{X}(\widetilde t)=0$, it follows from Corollary \ref{Cor1} that there exists either an $\widetilde A \in K$ such that $\widetilde A$ is functionally independent with $\widetilde t$ (resp. ${\rm ln}\widetilde t=\ell v+{\rm ln}h_a$) and is a first integral of $\mathcal{X}$, or  two functionally independent first integrals of $\mathcal{X}$. Consequently, the proof is completed.

Now we consider the case that $t$ cannot divide all $Q_j$'s. We assume without loss that $t$ is not a factor of $Q_1$.  Set $\deg Q_1=q_1\geq 1$, and
\begin{equation}\label{dengshi e1}
 Q_1=t^{q_1}+A_{q_1-1}t^{q_1-1}+\cdots+A_1t+A_0,
\end{equation}
where $A_0,\ldots,A_{q_1-1}\in K$, with $A_0\ne 0$ by the assumption. Taking the derivation $\mathcal{X}$ acting on $Q_1$ yields
\begin{equation}\label{dengshi e2}
    q_1\mathcal{X}(v)Q_1=\mathcal{X}(Q_1)={q_1}\mathcal{X}(v)t^{q_1}+(\mathcal{X}(A_{q_1-1})+(q_1-1)A_{q_1-1}\mathcal{X}(v))t^{q_1-1}+\cdots+\mathcal{X}(A_0).
\end{equation}
By the transcendence of $t$, this equality forces that
\begin{equation}\label{dengshi e3}
q_1\mathcal{X}(v)A_j=\mathcal{X}(A_j)+jA_j\mathcal{X}(v),\quad j=0,\ldots,q_1-1.
\end{equation}
Particularly, we have $q_1\mathcal{X}(v)A_0=\mathcal{X}(A_0)$. Furthermore, $\mathcal{X}(A_0)\ne 0$ (if not, $\mathcal{X}(v)=0$, a contradiction), and so $A_0$ is a rational function. Associated with \eqref{dengshi e3}, we know that $\mathcal{X}(\widetilde t)=0$, with $\widetilde t={A_0}^{-1}e^{q_1v}$. Note that $\widetilde t$ is transcendental over $K$ and, by Proposition \ref{extension gonghsi} and Lemma \ref{tr gongshi}, $K(\widetilde t) (\subseteq K(t))\subset L $ is an algebraic extension. By Corollary \ref{Cor1}, similar discussion as that in Subcase 2.1 shows that either there exists some rational first integral $\widetilde A$ for $\mathcal{X}$, which is functionally independent of $\widetilde t$, or $\mathcal{X}$ has two functionally independent rational first integrals.

It completes the proof of the proposition.
}
\end{proof}

{We remark that if $k=1$, by our methods, we can prove Theorem $A$ as well.}
Combining with Propositions \ref{Prop1}, \ref{Prop2} and \ref{Prop3}, we have completed the proof of Theorem \ref{Thm Stronger} for the elementary field extension $K\subset K(t)\subset L$ and the existence of two functionally independent first integrals in $L$ of $\mathcal X$.
	
Next we prove Theorem \ref{Thm Stronger} also for the elementary field extension $K\subset K(t)\subset L$ but with $k>2$ functionally independent first integrals in $L$.

Let $\Theta := \left\{ {{f_1}\left( x \right), \cdots ,{f_s}\left( x \right)} \right\}$ be a set of  smooth functions defined on ${\mathbb C}^n$ except perhaps a zero Lebesgue measure subset. We call $ \{{f_{{k_1}}}, \cdots ,{f_{{k_\ell }}}\}\subset \Theta$ a \textit{maximal functionally independent group} of $\Theta$, {if $f_{k_1}, \cdots ,f_{k_\ell }$ are functionally independent on ${\mathbb C}^n$, i.e. ${\rm d}f_{k_1}\wedge\cdots\wedge {\rm d}f_{k_\ell} \ne 0$ while any other $f \in \Theta$ satisfies ${\rm d}f\wedge{\rm d}f_{k_1}\wedge\cdots\wedge {\rm d}f_{k_\ell} = 0$.}

\begin{prop}\label{Prop4}
Consider the tower of the differential field extension $K\subset K(t) \subset L$, where $t$ is transcendental over $K$ and $L/K(t)$ is algebraic extension. Let $H_1$, $H_2$, $\ldots$, $H_k$ $($$k\geq2$$)$ be the functionally independent first integrals of $\mathcal{X}$ in $L$.
If $\mathcal{X}(t)=0$, then Theorem \ref{Thm Stronger} holds.
\end{prop}
\begin{proof}
Let
\begin{equation}\label{dengshi k1}
H_j^{{n_j}} + {a_{j,{n_j} - 1}}H_j^{{n_j} - 1} +  \cdots  + {a_{j,1}}{H_j} + {a_{j,0}} = 0, \quad j\in\{1,\ldots,k\},
\end{equation}
with ${a_{j,0}}, \ldots ,{a_{j,{n_j} - 1}}\in K(t)$, be the minimal algebraic  equation of $H_j$ with $j\in\{1,\ldots,k\}$. According to the proof of Lemma \ref{algebraic lemma}, one gets that
\[
\Theta_1 :=\left\{ {{a_{1,0}}, \ldots ,{a_{1,{n_1} - 1}}, \ldots ,{a_{k,0}}, \ldots ,{a_{k,{n_k} - 1}}} \right\}
\]
is contained in $C(K(t),\mathcal{X})$. Denote by $a_1,\ldots,a_s$ the maximal functionally independent group of $\Theta_1$. {We claim that $k \leq s$. Indeed, the equations in \eqref{dengshi k1} show that
\begin{equation}\label{dengshi k1 df}
{\rm d}H_j=-\frac{1}{\gamma_j}(H_j^{{n_j} - 1}{\rm d}{a_{j,{n_j} - 1}} +  \cdots  + {H_j}{\rm d}{a_{j,1}} + {\rm d}{a_{j,0}}),
\end{equation}
where $0 \ne \gamma_j=n_jH_j^{{n_j-1}}+(n_j-1){a_{j,{n_j} - 1}}H_j^{{n_j} - 2}+\cdots+a_{j,1}$, $j=1,\ldots,k$. By contrary, if $s<k$, it follows that any $k$ elements $a_{n_1},\ldots,a_{n_k}$ of $\Theta_1$ satisfy ${\rm d}a_{n_1}\wedge\cdots\wedge{\rm d}a_{n_k}=0$. However, the equality \eqref{dengshi k1 df} implies that ${\rm d}H_1\wedge\cdots\wedge{\rm d}H_k$ can be written as linear combinations of the $k-$forms ${\rm d}\widetilde {a}_1 \wedge \cdots \wedge {\rm d}\widetilde {a}_k$, where $\widetilde{a}_1,\ldots,\widetilde{a}_k \in \Theta_1$, and so ${\rm d}H_1\wedge\cdots\wedge{\rm d}H_k=0$, a contradiction. The claim follows.}
Particular, $s=n-1$ whenever $k=n-1$, otherwise, if $s=n$, the vector field $\mathcal X$ has $n$ functionally independent first integral, it results in the vector field $\mathcal X$ being null. Hereafter, for convenience, we take $s=k$ so that ${\rm d}a_{1}\wedge\cdots\wedge{\rm d}a_{k}\ne 0$.



Consider $a_j$, $j=1,\ldots,k$, and factorize them as
\begin{equation}\label{dengshi K2}
{a_j} = {h_{{a_j}}}\prod\limits_{i = 1}^{{m_j}} {Q_{j,i}^{{r_{j,i}}}},
\end{equation}
where $h_{a_1},\ldots, h_{a_k}\in K$, and for each given $j\in\{1,\ldots,k\}$, $Q_{j,1},\ldots,Q_{j,m_j}$ are pairwise distinct irreducible polynomials in $K[t]$ and $r_{j,1},\ldots,r_{j,m_j}\in \mathbb Z \setminus \{0\}$.
Substituting the expressions in \eqref{dengshi K2} into $\mathcal{X}(a_j)=0$, $j=1,\ldots,k$, yields
\begin{equation}\label{dengshi K3}
\prod\limits_{i = 1}^{{m_j}} {{Q_{j,i}}} \left(\frac{\mathcal{X}({h_{{a_j}}})}{h_{a_j}} + \sum\limits_{i = 1}^{{m_j}} {{r_{j,i}}\frac{{\mathcal{X}({Q_{j,i}})}}{{{Q_{j,i}}}}}\right)  = 0,\quad j=1,\ldots,k.
\end{equation}

Note that the left hand sides of these equations 
are polynomials in $K[t]$. This forces that $Q_{j,i}|\mathcal{X}(Q_{j,i})$.
Set
\[
\Theta_2:=\left\{ {h_{a_j},Q_{j,i}\left| {j = 1, \ldots ,k, \ i = 1, \ldots ,m_j} \right.} \right\}.
\]
Denote by $b_1,\ldots,b_{\widetilde s}$  \,  the maximal functionally independent group of $\Theta_2$. Then, by the functional independence of $a_1,\ldots,a_k$, it follows that $\widetilde s \geq k$. Without loss of generality and for simplifying notations, we set $\widetilde s =k$.

Since $\mathcal{X}(t)=0$ and $Q_{j,i}\in K[t]$ is monic, using the similar arguments as in the proof of Proposition \ref{Prop2} together with \eqref{dengshi K3}, we get that $\Theta_2 \subset C(K(t),\mathcal{X})$, and so $b_1,\ldots,b_{k}$ are functionally independent first integrals of the vector field $\mathcal X$ defined in \eqref{vector field}.
If $\{b_1,\ldots,b_k\}=\{h_{a_1},\ldots,h_{a_k}\}$, we are done. If $\{b_1,\ldots,b_k\}\ne\{h_{a_1},\ldots,h_{a_k}\}$, without loss of generality we set ${\rm deg}b_1 \geq 1$, and
\begin{equation}\label{dengshi K5}
{b_j} = {t^{{\alpha _j}}} + {B_{j,{\alpha _j} - 1}}{t^{{\alpha _j} - 1}} +  \ldots  + {B_{j,1}}t + {B_{j,0}},\quad j=1,\ldots,k.
\end{equation}
Taking the derivation $\mathcal{X}$ acting on the both sides of equation \eqref{dengshi K5} gives
\[
\{B_{1,0},\ldots,B_{1,\alpha_1-1},\ldots,B_{k,0},\ldots,B_{k,\alpha_k-1}\}\subset C(K,\mathcal{X}),
\]
where we have used the facts that $\mathcal X(b_j)=0$, and $t$ is transcendental over $K$.
Define the set $\Theta_3:=\{t,B_{1,0},\ldots,B_{1,\alpha_1-1},\ldots,B_{k,0},\ldots,B_{k,\alpha_k-1}\}$. Let $\{t, E_1,\ldots,E_{s'-1}\}$ be the maximal functionally independent group of the set $\Theta_3$. Then  $b_1,\ldots, b_k$ can be locally expressed as functions in terms of $t, E_1,\ldots,E_{s'-1}$ via \eqref{dengshi K5}. It implies $s'\geq k$, and so $t,E_1,\ldots,E_{k-1}$ are $k$ functionally independent first integrals of $\mathcal X$ that we desire (if $t=e^v$, we take $v$ replacing $t$).
The proof is completed.
\end{proof}

Summarizing the proof of Proposition \ref{Prop4} arrives the next conclusion.

\begin{Cor}\label{Cor2}
Let $K\subset K(\eta) \subset L$, with $K=\mathbb C(x_1,\ldots,x_n)$, $\eta$ transcendental over $K$, and $L$ an algebraic extension of $K(\eta)$. Assume that $\eta$ is a first integral of the vector field $\mathcal X$. 
If $H_1,\ldots,H_k \in C(L,\mathcal{X})$ $($$k\geq 2$$)$ are functionally independent, then either $\mathcal X$ has $k$ functionally independent rational first integrals, or there exist $\widetilde H_1,\ldots \widetilde H_{k-1}$ in $C(K,\mathcal{X})$ such that $\eta,\widetilde H_1,\ldots \widetilde H_{k-1}$ are functionally independent.
\end{Cor}

Now we turn to study the case that the transcendental element is not a first integral of $\mathcal X$.

\begin{prop}\label{Prop5}
Consider the tower of the elementary field extension $K\subset K(t) \subset L$ with $t$ transcendental over $K$ and $L$ an algebraic extension of $K(t)$. Let $H_1$, $H_2$, $\ldots$, $H_k$ $($$k\geq2$$)$ be the functionally independent first integrals of $\mathcal{X}$ in $L$. If $\mathcal{X}(t)\ne0$, then Theorem \ref{Thm Stronger}
holds.
\end{prop}
\begin{proof}
With the same notations as in the proofs of Propositions \ref{Prop3} and \ref{Prop4},  set $q_{j,i}:={\rm deg}Q_{j,i}$ for  $j=1,\ldots,k$, $i=1,\ldots,m_j$.

\noindent {\it Case }1. $t={\rm ln}v$. Since $Q_{j,i}$'s are monic and $\mathcal{X}(t)=\mathcal{X}(v)/v$, some direct computations show that ${\rm deg}Q_{j,i}>q_{j,i}-1\geq {\rm deg}\mathcal{X}(Q_{j,i})$. Since $Q_{j,i}$, $\mathcal{X}(Q_{j,i})\in K[t]$ and $Q_{j,i}|\mathcal{X}(Q_{j,i})$, it forces that $\mathcal{X}(Q_{j,i})=0$ for all $j=1,\ldots,k$, $i=1,\ldots,m_j$.
Then, the equality \eqref{dengshi K3} induces that $\{h_{a_1},\ldots,h_{a_k}\}\subset C(K,\mathcal{X})$, and so $\{b_1,\ldots,b_k\}\subset C(K(t),\mathcal{X})$.

If $\{b_1,\ldots,b_k\}=\{h_{a_1},\ldots,h_{a_k}\}$, i.e. $h_{a_1},\ldots,h_{a_k}$ are functionally independent rational first integrals of the vector field $\mathcal X$, we are done.

If $\{b_1,\ldots,b_k\}\ne \{h_{a_1},\ldots,h_{a_k}\}$,  we assume without loss of generality that ${\rm deg}b_1\geq 1$. Associated with the equality \eqref{dengshi K5}, substituting the expression of $b_1$ into $\mathcal X(b_1)=0$ yields
\begin{equation}\label{dengshi K6}
0=\mathcal{X}(b_1)=\left(\mathcal{X}({B_{1,{\alpha _1} - 1}}) + {\alpha _1}\mathcal{X}(v)/v\right){t^{{\alpha _1} - 1}} +  \ldots  + \left(\mathcal{X}({B_{1,0}}) + {B_{1,1}}\mathcal{X}(v)/v\right).
\end{equation}
The transcendence of $t$ induces that $\mathcal{X}(B_{1,\alpha_1-1}+\alpha_1{\rm ln}v)=0$. Set $\widetilde t=B_{1,\alpha_1-1}+\alpha_1{\rm ln}v$. Then $\widetilde t$ is also transcendental over $K$, and $K(\widetilde t)(=K(t))\subset L$. Hence, $L$ is also an algebraic extension of $K(\widetilde t)$. Applying Corollary \ref{Cor2} can complete the proof.

{
\noindent {\it Case }2. $t=e^v$. Since $Q_{j,i}$'s are monic and $\mathcal{X}(t)/t=\mathcal{X}(v)$, some direct computations show that  ${\rm deg}Q_{j,i}={\rm deg}\mathcal{X}(Q_{j,i})$. By $Q_{j,i}|\mathcal{X}(Q_{j,i})$ implies that $q_{j,i}\mathcal{X}(v)Q_{j.i}=\mathcal{X}(Q_{j,i})$.

\noindent{Subcase} 2.1. $\{a_1,\ldots,a_k\}=\{h_{a_1},\ldots,h_{a_k}\}$, which means that $h_{a_1},\ldots, h_{a_k}$ are functionally independent rational first integrals of $\mathcal{X}$, we are done.

\noindent{Subcase} 2.2. $\{a_1,\ldots,a_k\}\ne\{h_{a_1},\ldots,h_{a_k}\}$. We assume without loss of generality that $a_1 \ne h_{a_1}$.

If, for all $i=1,\ldots,m_1$, $t$ divides $Q_{1,i}$, then $a_1=h_{a_1}t^{\ell}$, for some non-zero integer $\ell$. By assumption that $\mathcal{X}(t)\ne 0$, we have $\mathcal{X}(h_{a_1})\ne 0$. Therefore, taking $\widetilde t=h_{a_1}e^{\ell v}$, $\widetilde t$ is a transcendental first integral of the vector field, so is ${\rm ln}\widetilde t=\ell v +{\rm ln}h_{a_1}$. By Lemma \ref{tr gongshi}, $K(t)$ is algebraic over $K(\widetilde t)$, and so is $L$. The remaining proof follows from Corollary \ref{Cor2}.

If some $Q_{1,j}$ cannot be divided by $t$, we set that
${\rm deg}Q_{1,1}=q_{1,1}\geq 1$, $t$ is not a factor of $Q_{1,1}$ and
\begin{equation}\label{dengshi e1}
Q_{1,1}=t^{q_{1,1}}+A_{q_{1,1}}t^{q_{1,1}-1}+\cdots+A_1t+A_0,
\end{equation}
where $A_{q_{1,1}},\ldots,A_0 \in K$, and by the assumption, $A_0 \ne 0$.
Combining with $q_{1,1}\mathcal{X}(v)Q_{1.1}=\mathcal{X}(Q_{1,1})$ yields that
\begin{equation}\label{dengshi e2}
q_{1,1}\mathcal{X}(v)Q_{1,1}=(q_{1,1}\mathcal{X}(v))t^{q_{1,1}}+(\mathcal{X}(A_{q_{1,1}-1}+(q_{1,1}-1)A_{q_{1,1}-1}\mathcal{X}(v))t^{q_{1,1}-1}+\cdots+\mathcal{X}(A_0),
\end{equation}
which shows that $q_{1,1}\mathcal{X}(v)A_0=\mathcal{X}(A_0)$. Note that, $\mathcal{X}(A_0)\ne 0$, otherwise, $\mathcal{X}(v)=0$, a contradiction. Taking $\widetilde t={A_0}^{-1}e^{q_{1,1} v}$. Then, $\widetilde t$ is a transcendental first integral of the vector field $\mathcal{X}$, and $K(t)$ is algebraic over $K(\widetilde t)$, and so is $L$. Replacing $\eta$ in Corollary \ref{Cor2} by $\widetilde t$ reaches the result that we want to prove.

It completes the proof of the proposition.}
\end{proof}

Up to now we have proved Theorem \ref{Thm Stronger} for $K\subset K(t)\subset L$ with $t$ transcendental over $K$ and $L$ an algebraic extension of $K(t)$.
Next we prove Theorem \ref{Thm Stronger} via induction on the transcendental degree tr$[L:K]$ of the elementary field extension $L/K$. Some ideas on induction follows from \cite{RischRObert,Prelle-Singer}.
Assume that $H_1,\ldots,H_k$, $1\leq k \leq n-1$, are $k$ functionally  independent first integrals of the vector field $\mathcal X$ in $L$.

If ${\rm tr}[L:K]=1$, the proof follows from Propositions \ref{Prop1}, \ref{Prop4} and \ref{Prop5}. Here we have used the fact that the class of algebraic extensions is distinguished (see \cite[Section 5.1, Proposition 1.7]{Lang}).

To apply induction on the transcendental degree ${\rm tr}[L:K]$, we assume that Theorem \ref{Thm Stronger} holds when ${\rm tr}[L:K]=m-1$ for $2 \leq m \in \mathbb N$.

For ${\rm tr}[L:K]=m$, we claim that the elementary field extension $L/K$ can be written without loss in the tower form
\[
K_0:=K\subset K_1:=K(t_1)\subset K_2:=K_1(t_2)\subset \cdots \subset K_N:=K_{N-1}(t_{N}) \subset L,
\]
where $t_j$ is either an algebraic element or a transcendental element over $K_{j-1}$ for $j=1,\ldots,N$, and $L$ is algebraic over $K(t_N)$ ($N\geq m$). {Note that, when $L$ is not algebraic over $K_{N-1}(t_N)$, the definition of the tower of the elementary field extension shows that $L$ will be written as $K_{N}(t_{N+1})$ with transcendental element $t_{N+1}$ of $K_{N}=K_{N-1}(t_N)$, which is still akin to the assumption since $L=K_{N}(t_{N+1})$ is algebraic over itself }.

For convenience, we assume without loss that $t_1$ is transcendental over $K$.

Combining with Lemma \ref{tr gongshi}, direct calculations show that tr$[L:K(t_1)]=m-1$. Replacing $K$ with $K(t_1)$ in Theorem \ref{Thm Stronger}. By the induction assumption on ${\rm tr}[L:K(t_1)]=m-1$, we know that there exist
\[
\widetilde{H}_j={\widetilde w_{0,j}} + \sum\limits_{i= 1}^{{\widetilde m_j}} {\widetilde{c_{i,j}}\ln {\widetilde w_{i,j}}},\quad j=1,\ldots,k,
\]
with $\widetilde w_{0,j},\ldots, \widetilde w_{\widetilde m_j,j}\in \overline{K(t_1)}\cap L$,  and $\widetilde c_{0,j},\ldots,\widetilde c_{\widetilde m_j,j} \in \mathbb C$, $j=1,\ldots,k$,
such that $\widetilde{H}_1,\ldots, \widetilde{H}_k$ are $k$ functionally independent first integrals of the vector field $\mathcal X$.
Recall that $\overline{K(t_1)}$ is the algebraic closure of $K(t_1)$.
{Set $\widetilde K:=K(t_1,\widetilde w_{0,1},\ldots, \widetilde w_{\widetilde m_k,k})$. It is obvious that $\widetilde K \subset L$, and $K (\subset K(t_1)) \subset \widetilde K$ is an elementary extension. Moreover, Lemma \ref{tr gongshi} implies} that tr$[\widetilde K:K]$=tr$[\widetilde K : K(t_1)]$+tr$[K(t_1):K]=1$, where we have used the fact that tr$[\widetilde K:K(t_1)]$=0.
Replacing $L$, in Theorem \ref{Thm Stronger}, with $\widetilde K$, induces that $\widetilde{H}_1 ,\ldots, \widetilde{H}_k$ satisfy the conditions in Theorem \ref{Thm Stronger}, and therefore applying Theorem \ref{Thm Stronger} to $\widetilde{H}_1,  \ldots,\widetilde{H}_k$ reaches that there exist
\[
\overline{H}_j={w_{0,j}} + \sum\limits_{i = 1}^{{m_j}} {{c_{i,j}}\ln {w_{i,j}}},\quad j=1,\ldots,k,
\]
with $ w_{0,j},\ldots, w_{m_j,j}\in \overline{K} \cap \widetilde K \subset \overline{K}\cap L$, and $c_{0,j},\ldots,c_{m_j,j} \in \mathbb C$, $j=1,\ldots,k$, such that the vector field $\mathcal X$ has the $k$ functionally independent first integrals  $\overline{H}_1,\ldots,\overline{H}_k$, which are of the desired form.
Recall again that $\overline K$ is the algebraic closure of $K$.

It completes the proof of the theorem. \qed

Having proved Theorem \ref{Thm Stronger} also verifies Theorem \ref{Theorem1}.
With this theorem in hand, we can prove Theorem \ref{Theorem2}.

\section{Proof of Theorem \ref{Theorem2}}\label{S3}

In the proof of Theorem \ref{Theorem2} we need more knowledge from the theories of  field extensions and of Galois groups than those in Section \ref{Pre}. For reader's convenience and completeness, we introduce some necessary ones here, for a reference, see e.g. \cite[Sections 5.3, 5.4, 6.1]{Lang}.
Consider an algebraic field extension $F/E$. An element $\beta\in F$ is said to be {\it separable} over $E$ if the characteristic polynomial of $\beta$ has no multiple  zeros in $F$. If all the elements of $F$ is separable over $E$, we call $F$ {\it separable over $E$} or say that {\it the extension $F/E$ is separable}.

We call an algebraic field extension $F/E$  {\it normal} if each {irreducible} polynomial of $E[X]$ which has a zero in $F$ splits into product of linear factors in $F[X]$. The {\it normal closure} of an algebraic field extension $F/E$ is a field extension $\overline{F}^N$ of $F$ such that $\overline{F}^N/E$ is normal and $\overline{F}^N$ is the minimal field extension satisfying this property. {See \cite[Section 6.1]{Lang} for the existence of normal closure of a given field.}

A bijective map $\phi: E \longrightarrow E$ is called a {\it field automorphism} if it satisfies, for any $\alpha$, $\beta$ in $E$, $\phi(\alpha+\beta)=\phi(\alpha)+\phi(\beta)$ and $\phi(\alpha {\beta})=\phi(\alpha){\phi(\beta)}$. We say that $\phi$ is a field automorphism of $F$ over $E$, if it fixes all the elements of $E$. The set of all field automorphisms of $F$ over $E$ forms a group under the composition of maps. This group is called a {\it Galois group}, denoted by Gal$(F/E)$. Refer to \cite[Section 6.1, Corollary 1.6]{Lang}, when $F/E$ is an algebraic extension, Gal$(\overline{F}^N/E)$ is finite. For more information and results related to the notions mentioned above, see e.g. \cite{Lang}.

To prove Theorem \ref{Theorem2} we first characterize properties of the derivatives of the elements in an algebraic extension.   Recall again that for a differential field $K$, $\overline K$ is the algebraic closure of $K$.

\begin{lem} \label{Lemma3.1}
Consider the differential field $(K,\Delta)$ with $K=\mathbb C(x)$, $x=(x_1,\ldots,x_n)$ and $\Delta=\{\partial_{x_1},\ldots,\partial_{x_n}\}$.
If $\omega$ is an algebraic element, i.e. $\omega \in \overline{K}$, then for any $\delta \in \Delta$, $\delta \omega$ is also an algebraic element over $K$.
\end{lem}
\begin{proof} In \cite[pp. 132--133]{Zhang2017} there exists a proof for two dimensional case. Here for completeness and it is short, we present a proof here.
By the assumption, let
\begin{equation}\label{minimal Eq Sec2}
{X^m} + {\beta _{m - 1}}{X^{m - 1}} +  \ldots  + {\beta _1}X + {\beta _0}
\end{equation}
be the minimal algebraic equation of $\omega$, where $\beta_0,\beta_1,\ldots,\beta_{m-1} \in K$. Since $\beta_i$ is rational, so is $\delta \beta_i$ for all $\delta \in \Delta$, $i=0, 1,\ldots,m-1$. Denote by $\widetilde K$ the minimal algebraic extension containing $\omega$ of $K$. Differentiating the equation \eqref{minimal Eq Sec2} with respect to $x_j$, $j\in \{1,\ldots,n\}$, gives
\[{\partial _{{x_j}}}\omega  =  - \frac{{{\omega ^{m - 1}}{\partial _{{x_j}}}{\beta _{m - 1}} +  \cdots  + {\partial _{{x_j}}}{\beta _1}\omega  + {\partial _{{x_j}}}{\beta _0}}}{{m{\omega ^{m - 1}} + (m - 1){\beta _{m - 1}}{\omega ^{m - 2}} +  \cdots  + {\beta _1}}}.
\]
Note that the denominator of the right hand side of this last equation does not identically vanish. This shows that ${\partial _{{x_j}}}\omega \in \overline K$.
\end{proof}

\begin{proof}[Proof of Theorem $\ref{Theorem2}$]
If ${\rm div}\mathcal{X}=0$, the theorem is trivial. We consider the case ${\rm div}\mathcal{X}\not\equiv 0$. By definition, a Jacobian multiplier of $\mathcal X$ is a nonconstant smooth function $J$ satisfying ${\rm div}(J\mathcal X)\equiv 0$, i.e. $\mathcal X(J)=-J\mbox{\rm div}\mathcal X$.

By the assumption of the theorem, one gets from Theorem \ref{Theorem1} that the vector field $\mathcal X$ has the functionally independent first integrals of the form
\begin{equation}\label{dengshi21}
{H}_j = {u_{0,j}} + \sum\limits_{i = 1}^{{n_j}} {{c_{i,j}}\ln {u_{i,j}}},\quad j=1,\ldots,n-1,
\end{equation}
where $u_{0,j},\ldots ,u_{n_j,j}\in \overline{K}\cap L $, $j=1,\ldots,n-1$, which are defined in $\Omega:=\mathbb C^n\setminus \Omega_0$, with $\Omega_0$ a zero Lebesgue measure subset of $\mathbb C^n$. Taking any derivation $\delta \in \Delta$ acting on the equality \eqref{dengshi21} gives
\[
\delta H_j=\delta{u_{0,j}} + \sum\limits_{i = 1}^{{n_j}} {c_{i,j}}\frac{\delta{u_{i,j}}}{{{u_{i,j}}}}.
\]
Let $F$ be the minimal algebraic extension containing the set $\{u_{0,j},\ldots ,u_{n_j,j}$:\ \  $j=1,\ldots,n-1$\} of $K=\mathbb C(x_1,\ldots,x_n)$.
By Lemma \ref{Lemma3.1}, we know that for all $\delta \in \Delta$, $\delta H_j\in F$.
To simplify notations, in what follows we denote $\partial_{x_s}$ by $\partial_s$ and $(\partial_1,\ldots,\partial_n)$ by $\partial$.
The first integrals $H_j$'s satisfy the equations
\begin{equation}\label{dengshi22}
\mathcal{X}(H_j)(x) \equiv 0,\quad x\in \Omega,\quad j=1,\ldots,n-1.
\end{equation}
By the functional independence of $H_1,\ldots,H_{n-1}$, without loss of generality, we set
\[
\Lambda(x):={\rm det}(\partial_1\mathcal{H}_1(x),\ldots,\partial_{n-1}\mathcal{H}_{n-1}(x))\ne 0,\quad x\in \Omega,
\]
where
\[
\begin{split}
\mathcal{H}(x):&=(H_1(x),\ldots,H_{n-1}(x))^{T},\\
\partial_i\mathcal{H}(x):&=(\partial_iH_1(x),\ldots,\partial_iH_{n-1}(x))^{T},\quad i=1,\ldots,n.
\end{split}
\]
with $T$ representing the transpose of a matrix.
For $s=1,\ldots,n-1$, define
\[
\Lambda_s(x) :={\rm det}(\partial_1\mathcal{H}(x),\ldots,\partial_{s-1}\mathcal{H}(x),\partial_n\mathcal{H}(x),
\partial_{s+1}\mathcal{H}(x),\ldots,\partial_{n-1}\mathcal{H}(x)).
\]
By the properties on $\partial_s H_j$ obtained above, it follows easily that $\Lambda$ and $\Lambda_s$, $s=1,\ldots,n-1$, belong to $F$.

Since
\begin{align*}
\partial_n \Lambda = &\sum\limits_{s = 1}^{n - 1} {\det \left( {{\partial _1}\mathcal{H}, \ldots ,{\partial _{s - 1}}\mathcal{H},{\partial _n}{\partial _s}\mathcal{H},{\partial _{s + 1}}\mathcal{H}, \ldots ,{\partial _{n - 1}}\mathcal{H}} \right)},\\
\partial_s\Lambda_s = & \det \left( {{\partial _1}\mathcal{H}, \ldots ,{\partial _{s - 1}}\mathcal{H},{\partial _s}{\partial _n}\mathcal{H},{\partial _{s + 1}}\mathcal{H}, \ldots ,{\partial _{n - 1}}\mathcal{H}} \right),\\
&+\sum\limits_{\ell = 1,\ell  \ne s}^{n - 1} \det \left( {\partial _1}\mathcal{H}, \ldots ,{\partial _{\ell  - 1}}\mathcal{H},{\partial _s}{\partial _\ell }\mathcal{H},{\partial _{\ell  + 1}}\mathcal{H}, \ldots ,{\partial _{s - 1}}\mathcal{H},\right.\\
&\qquad\qquad \qquad \qquad  \left.{\partial _n}\mathcal{H},{\partial _{s + 1}}\mathcal{H}, \ldots ,{\partial _{n - 1}}\mathcal{H} \right),
\end{align*}
we have
\begin{align*}
 \partial_1\Lambda_1 & +\ldots+\partial_{n-1}\Lambda_{n-1}-\partial_n\Lambda\\
=&\sum\limits_{s = 1}^{n - 1} \sum\limits_{\ell  = 1,\ell  \ne s}^{n - 1} \det \left( {\partial _1}\mathcal{H}, \ldots ,{\partial _{\ell  - 1}}\mathcal{H},{\partial _s}{\partial _\ell }\mathcal{H},{\partial _{\ell  + 1}}\mathcal{H}, \ldots ,\right.\\
&\qquad\qquad \qquad \qquad \left.{\partial _{s - 1}}\mathcal{H},{\partial _n}\mathcal{H},{\partial _{s + 1}}\mathcal{H}, \ldots ,{\partial _{n - 1}}\mathcal{H} \right) \\
=&0,
\end{align*}
where in the last equality we have used the fact that
\begin{align*}
	&\det \left( {{\partial _1}\mathcal{H}, \ldots ,{\partial _{\ell  - 1}}\mathcal{H},{\partial _s}{\partial _\ell }\mathcal{H},{\partial _{\ell  + 1}}\mathcal{H}, \ldots ,{\partial _{s - 1}}\mathcal{H},{\partial _n}\mathcal{H},{\partial _{s + 1}}\mathcal{H}, \ldots ,{\partial _{n - 1}}\mathcal{H}} \right)\\
	& =  - \det \left( {{\partial _1}\mathcal{H}, \ldots ,{\partial _{\ell  - 1}}\mathcal{H},{\partial _n}\mathcal{H},{\partial _{\ell  + 1}}\mathcal{H}, \ldots ,{\partial _{s - 1}}\mathcal{H},{\partial _\ell}{\partial _s }\mathcal{H},{\partial _{s + 1}}\mathcal{H}, \ldots ,{\partial _{n - 1}}\mathcal{H}} \right).
	\end{align*}
Solving the equation \eqref{dengshi22} by the Cramer's rule gives
\[
{P_s}({x_1}, \ldots ,{x_n}) =  - \frac{\Lambda_s}{\Lambda}{P_n}({x_1}, \ldots ,{x_n}),\quad s=1,\ldots,n-1.
\]
Recall that $P_s$'s are the components of the vector field $\mathcal X$, defined in  \eqref{vector field}.
Set $h(x):=P_n/ \Lambda=-P_s/ \Lambda_s$, $s=1,\ldots,n-1$. Then $h\in F$. Direct calculations show that
\begin{align*}
\mathcal{X}(h)&=P_1\partial_1h+\ldots P_n\partial_nh
\\
& =P_1\partial_1\left(\frac{-P_1}{\Lambda_1}\right)+\ldots +P_{n-1}\partial_1\left(\frac{-P_{n-1}}{\Lambda_{n-1}}\right)+P_n\partial_n\left(\frac{P_n}{\Lambda}\right)\\
&=h(\partial_1(P_1)+\ldots+\partial_n(P_n))+h^2(\partial_1\Lambda_1+\ldots+\partial_{n-1}\Lambda_{n-1}-\partial_n\Lambda)=h{\rm div}(\mathcal{X}).
\end{align*}
This shows that $J=h^{-1}\in F$ and it is a Jacobian multiplier of the vector field $\mathcal X$, because
\begin{equation}\label{Jacobian multiplier}
 \mathcal{X}(J)=-J{\rm div}(\mathcal{X}).
\end{equation}

Let $\overline{F}^N$ be the normal closure of $F$, and $\mathcal{G}$ the Galois group formed by the automorphisms of $\overline{F}^N$ fixing $K$. According to Lang \cite[Chapter VI, Theorem 1.8]{Lang}, proved by Artin, the group $\mathcal{G}$ is of finite order, which is denoted by $N$. Moreover, $N\leq [F:K]$, the degree of the algebraic extension of the field.
 Rosenlicht \cite{M.Rosenlicht} proved that the $K$--automorphisms of $\overline{F}^N$ commute with the derivations on $\overline{F}^N$. So, for any $\sigma \in \mathcal{G}$, we get from \eqref{Jacobian multiplier} that
\[
\mathcal{X}(\sigma J)=-\sigma(J){\rm div}\mathcal{X}.
\]
This further shows that
\[
\mathcal{X}(\widetilde J)=-\widetilde J {\rm div}\mathcal{X},
\]
where $\widetilde J=\frac{1}{N}{\rm trace}J$, with ${\rm trace} J:=\sum\limits_{\sigma  \in \mathcal{G}} {\sigma (J)}$.
Note that for any $\sigma_0 \in \mathcal{G}$, we have
\[\sigma_0(\widetilde J)=\frac{1}{N}\sigma_0({\rm trace}J)=\frac{1}{N}{\rm trace}J=\widetilde J.\]
This implies that $\widetilde J \in K=\mathbb C(x_1,\ldots,x_n)$. So far, we have proved that the vector field $\mathcal X$ has a rational Jacobian multiplier.

It completes the proof of the theorem.
\end{proof}


\section*{Acknowledgements}
This work is partially supported by National Key R$\&$D Program of China grant
number 2022YFA1005900.

The second author is also partially supported by the National Natural Science Foundation of China (NSFC) under Grant Nos. 12071284 and 12161131001. 


\vspace{0.1in}
\begin{thebibliography}{10}
\bibitem{ADDdaM2007}
J.~Avellar, L.~G.~S. Duarte, S.~E.~S. Duarte, and L.~A. C.~P. da~Mota.
\newblock A semi-algorithm to find elementary first order invariants of rational second order ordinary differential equations.
\newblock {\em Appl. Math. Comput.}, 184(1):2--11, 2007.

\bibitem{B.Jamil}
J.~Baddoura.
\newblock Integration in finite terms with elementary functions and dilogarithms.
\newblock {\em J. Symbolic Comput.}, 41(8):909--942, 2006.

\bibitem{Bruns1887}
H.~Bruns.
\newblock Uber die integrale des vielk\"{o}rper-problems.
\newblock {\em Acta Math.}, 11:25--96, 1887.

\bibitem{Llibre2003}
J.~Chavarriga, H.~Giacomini, J.~Gin\'{e}, and J.~Llibre.
\newblock Darboux integrability and the inverse integrating factor.
\newblock {\em J. Differential Equations}, 194(1):116--139, 2003.

\bibitem{Christopher1999}
C.~Christopher.
\newblock Liouvillian first integrals of second order polynomial differential equations.
\newblock {\em Electron. J. Differential Equations}, pages No. 49, 7, 1999.

\bibitem{CLY2014}
C.~Christopher, C.~Li, and J.~Torregrosa.
\newblock {\em Limit cycles of differential equations}.
\newblock Advanced Courses in Mathematics. CRM Barcelona. Birkh\"{a}user/Springer, Cham, second edition, 2024.

\bibitem{CLPW2019}
C.~Christopher, J.~Llibre, C.~Pantazi, and S.~Walcher.
\newblock On planar polynomial vector fields with elementary first integrals.
\newblock {\em J. Differential Equations}, 267(8):4572--4588, 2019.

\bibitem{Duarte}
L.~G.~S. Duarte and L.~A. C.~P. da~Mota.
\newblock An efficient method for computing {L}iouvillian first integrals of planar polynomial vector fields.
\newblock {\em J. Differential Equations}, 300:356--385, 2021.

\bibitem{Forsyth1900}
A.~R. Forsyth.
\newblock {\em Theory of differential equations. 1. {E}xact equations and {P}faff's problem; 2, 3. {O}rdinary equations, not linear; 4. {O}rdinary linear equations; 5, 6. {P}artial differential equations}.
\newblock Dover Publications, Inc., New York, 1959.
\newblock Six volumes bound as three.

\bibitem{J.Gine}
J.~Gin\'{e}.
\newblock Reduction of integrable planar polynomial differential systems.
\newblock {\em Appl. Math. Lett.}, 25(11):1862--1865, 2012.

\bibitem{Hodge}
W.~V.~D. Hodge and D.~Pedoe.
\newblock {\em Methods of algebraic geometry. {V}ol. {I}}.
\newblock Cambridge Mathematical Library. Cambridge University Press, Cambridge, 1994.
\newblock Book I: Algebraic preliminaries, Book II: Projective space, Reprint of the 1947 original.

\bibitem{Kus1983}
M.~K\'u{s}.
\newblock Integrals of motion for the {L}orenz system.
\newblock {\em J. Phys. A}, 16(18):L689--L691, 1983.

\bibitem{Lang}
S.~Lang.
\newblock {\em Algebra}, volume 211 of {\em Graduate Texts in Mathematics}.
\newblock Springer-Verlag, New York, third edition, 2002.

\bibitem{Llibre2004}
J.~Llibre.
\newblock Integrability of polynomial differential systems.
\newblock In {\em Handbook of differential equations}, pages 437--532. Elsevier/North-Holland, Amsterdam, 2004.

\bibitem{LlibreZhang2002}
J.~Llibre and X.~Zhang.
\newblock Invariant algebraic surfaces of the {L}orenz system.
\newblock {\em J. Math. Phys.}, 43(3):1622--1645, 2002.

\bibitem{Lorenz1963}
E.~Lorenz.
\newblock Deterministic nonperiodic flow.
\newblock {\em J. Atmos. Sci.}, 20:130--141, 1963.

\bibitem{Man-MacCallum}
Y.-K. Man and M.~A.~H. MacCallum.
\newblock A rational approach to the {P}relle-{S}inger algorithm.
\newblock {\em J. Symbolic Comput.}, 24(1):31--43, 1997.

\bibitem{Prelle-Singer}
M.~J. Prelle and M.~F. Singer.
\newblock Elementary first integrals of differential equations.
\newblock {\em Trans. Amer. Math. Soc.}, 279(1):215--229, 1983.

\bibitem{RischRObert}
R.~H. Risch.
\newblock The problem of integration in finite terms.
\newblock {\em Trans. Amer. Math. Soc.}, 139:167--189, 1969.

\bibitem{M.Rosenlicht}
M.~Rosenlicht.
\newblock On {L}iouville's theory of elementary functions.
\newblock {\em Pacific J. Math.}, 65(2):485--492, 1976.

\bibitem{Segur1982}
H.~Segur.
\newblock Soliton and the inverse acattering transform, in.
\newblock {\em Topics in Ocean-Physics, {\rm edited by A.R. Osborne and P. Malanotte Rizzoli (North-Holland, Amstersan)}}, pages 235--277, 1982.

\bibitem{Singer1992}
M.~F. Singer.
\newblock Liouvillian first integrals of differential equations.
\newblock {\em Trans. Amer. Math. Soc.}, 333(2):673--688, 1992.

\bibitem{S.Varad}
V.~R. Srinivasan.
\newblock Differential subfields of liouvillian extensions.
\newblock {\em J. Algebra}, 550:358--378, 2020.

\bibitem{Swinnerton-Dyer}
S.~Swinnerton-Dyer.
\newblock The invariant algebraic surfaces of the lorenz system.
\newblock {\em Math. Proc. Camb. Phil. Soc.}, 132(3):385–393, 2002.

\bibitem{Zhang2016}
X.~Zhang.
\newblock Liouvillian integrability of polynomial differential systems.
\newblock {\em Trans. Amer. Math. Soc.}, 368(1):607--620, 2016.

\bibitem{Zhang2017}
X.~Zhang.
\newblock {\em Integrability of dynamical systems: algebra and analysis}, volume~47 of {\em Developments in Mathematics}.
\newblock Springer, Singapore, 2017.

\bibitem{Zoladek1998}
H.~Zoladek.
\newblock The extended monodromy group and {L}iouvillian first integrals.
\newblock {\em J. Dynam. Control Systems}, 4(1):1--28, 1998.

\end{thebibliography}
\end{document}